\documentclass[12pt]{article}
\usepackage{amsmath,amsthm,amsfonts,amssymb}
\usepackage{fullpage}
\usepackage{multirow}
\usepackage{array}
\usepackage{bbm}
\usepackage[noblocks]{authblk}
\usepackage{verbatim}
\usepackage{pgf,tikz,pgfplots}
\usepackage{mathrsfs}
\usetikzlibrary{arrows}
\usepackage{float}
\usepackage{enumerate}
\usepackage{diagbox}
\usepackage{color, colortbl}
\usepackage{hhline}
\usepackage{hyperref}

\newtheorem{theorem}{Theorem}[section]
\newtheorem{lemma}[theorem]{Lemma}
\newtheorem{proposition}[theorem]{Proposition}

\theoremstyle{definition}

\newlength{\Oldarrayrulewidth}

\newcommand{\N}{\mathbb{N}}

\renewcommand{\mod}[2]{\equiv#1\textup{ (mod }#2\textup{)}}

\begin{document}

\title{Every Arithmetic Progression Contains Infinitely Many $b$-Niven Numbers}
\author[1]{Joshua~Harrington\thanks{joshua.harrington@cedarcrest.edu}}
\author[2]{Matthew~Litman\thanks{mclitman@ucdavis.edu.edu}}
\author[3]{Tony~W.~H.~Wong\thanks{wong@kutztown.edu}}
\affil[1]{Department of Mathematics, Cedar Crest College}
\affil[2]{Department of Mathematics, University of California Davis}
\affil[3]{Department of Mathematics, Kutztown University of Pennsylvania}
\date{\today}

\maketitle

\begin{abstract}
For an integer $b\geq 2$, a positive integer is called a $b$-Niven number if it is a multiple of the sum of the digits in its base-$b$ representation.  In this article, we show that every arithmetic progression contains infinitely many $b$-Niven numbers.

\end{abstract}

\section{Introduction}
Let $\N$ denote the set of positive integers and let $b\geq2$ be an integer. For all $n\in\N$ and $0\leq i\leq\lfloor\log_bn\rfloor$, let $\nu_b(n,i)$ be nonnegative integers such that $\nu_b(n,i)\leq b-1$ and $n=\sum_{i=0}^{\lfloor\log_bn\rfloor}\nu_b(n,i)b^i$. In other words, $\nu_b(n,i)$ is the $(i+1)$-st digit from the right in the base-$b$ representation of $n$. Furthermore, define $s_b:\N\to\N$ by $s_b(n)=\sum_{i=0}^{\lfloor\log_bn\rfloor}\nu_b(n,i)$. A positive integer $n$ is \emph{$b$-Niven} if $s_b(n)\mid n$. 

It was shown in 1993 by Cooper and Kennedy \cite{ck} that there are no 21 consecutive $10$-Niven numbers.  Their result was generalized in 1994 by Grundman \cite{g}, who showed that there are no $2b+1$ consecutive $b$-Niven numbers. In 1994, Wilson \cite{w} proved that for each $b$, there are infinitely many occurrences of $2b$ consecutive $b$-Niven numbers. These results were recently extended by Grundman, Harrington, and Wong \cite{ghw}, who investigated the maximum lengths of arithmetic progressions of $b$-Niven numbers.

In this article, we prove that every arithmetic progression contains infinitely many $b$-Niven numbers.

\section{Main Results}
The following lemma is sometimes referred to as the ``postage stamp theorem," the ``chicken McNugget theorem," or ``Frobenius' coin theorem."

\begin{lemma}\label{lem:frobenius}
Let $u$ and $v$ be integers with $uv\geq0$ and $\gcd(u,v)=1$. Then every integer $w$ such that $w$ shares the same sign with $u$ and $v$ and satisfies $|w|\geq(|u|-1)(|v|-1)$ can be written in the form $w=gu+hv$ for some nonnegative integers $g$ and $h$.
\end{lemma}

The following two lemmas, which will be useful in our proof, are easy exercise in elementary number theory.

\begin{lemma}\label{lem:d|b-1}
If $d\mid b-1$, then for all $u\in\N$, we have $d\mid u$ if and only if $d\mid s_b(u)$.
\end{lemma}

\begin{lemma}\label{lem:boundsb}
For all integers $2\leq n'\leq n$,
\begin{equation*}
s_b(n')\leq(b-1)\lceil\log_b(n)\rceil.
\end{equation*}
\end{lemma}

For positive integers $m$ and $r$, let
$$\mathcal{S}_{m,r}=\{mx+r:x\in\N\}.$$

\begin{proposition}\label{prop:existence}
Let $d=\gcd(s_b(m),s_b(r),b-1)$. If $\gcd(s_b(m),s_b(r))=d$, then $\mathcal{S}_{m,r}$ contains at least one $b$-Niven number.
\end{proposition}

\begin{proof}
Let $k_0(b,m,r)\in\N$ such that for all integers $k\geq k_0$, 
\begin{equation}\label{eqn:k0def}
k\geq(b-1)\left\lceil\log_b\left(\frac{s_b(m)}{d}\cdot k+\frac{s_b(r)}{d}\right)\right\rceil+(b-2)\left((b-1)\left\lceil\log_b\left(\frac{s_b(m)}{d}\cdot k+\frac{s_b(r)}{d}\right)\right\rceil-1\right).
\end{equation}
Note that $k_0$ is well-defined since $b,m,r$ are constants and the right hand side of equation~\eqref{eqn:k0def} is of order $O(\log k)$. Using Dirichlet's theorem on primes in arithmetic progressions, let $k\in\N$ be such that $k\geq\max\{k_0,b,m\}$ and $p=\frac{s_b(m)}{d}\cdot k+\frac{s_b(r)}{d}$ is a prime. Since $p>k\geq\max\{b,m\}$, we have $p\nmid bm$. Furthermore, let $\tilde{x}$ be the smallest positive integer such that $\tilde{x}\mod{-m^{-1}r}{p}$. Due to Lemma~\ref{lem:boundsb} and equation~\eqref{eqn:k0def}, we have
$$k\geq s_b(\tilde{x})+(b-2)(s_b(p)-1).$$
By Lemma~\ref{lem:frobenius}, there exist nonnegative integers $g$ and $h$ such that
$$k=s_b(\tilde{x})+g(b-1)+h\cdot s_b(p).$$

Let $\omega\in\N$ be a multiple of $p-1$ such that $b^\omega>\max\{m,r\}$. Note that $b^\omega\mod{1}{p}$ by Fermat's little theorem since $p\nmid b$. We now define a function $\tau_b:\N\to\N$ as follows. For each fixed $n\in\N$, let $\sigma_{-1}=0$ and $\sigma_i=\sum_{j=0}^i\nu_b(n,j)$ for $0\leq i\leq\lfloor\log_bn\rfloor$. Then
$$\tau_b(n)=\sum_{j=1}^{\sigma_{\lfloor\log_bn\rfloor}}b^{j\omega+\ell_j},$$
where $\ell_j=i$ for the unique $i\in\{0,1,2,\dotsc,\lfloor\log_bn\rfloor\}$ satisfying $\sigma_{i-1}<j\leq\sigma_i$. It is important to notice that the construction of $\tau_b(n)$ guarantees $s_b(\tau_b(n))=\sigma_{\lfloor\log_bn\rfloor}=s_b(n)$ and $\tau_b(n)\equiv\sum_{j=1}^{\sigma_{\lfloor\log_bn\rfloor}}b^{\ell_j}\equiv\sum_{i=0}^{\lfloor\log_bn\rfloor}\nu_b(n,i)b^i\mod{n}{p}$.

Let $x_0=\tau_b(\tilde{x})$, and for each positive integer $t\leq g$, let
$$x_t=x_{t-1}-b^{\lfloor\log_bx_{t-1}\rfloor}+\sum_{\iota=1}^bb^{\iota\omega+\lfloor\log_bx_{t-1}\rfloor-1}.$$
From this construction, we have $s_b(x_t)=s_b(x_{t-1})+b-1$ and 
$$x_t\equiv x_{t-1}-b^{\lfloor\log_bx_{t-1}\rfloor}+b\cdot b^{\lfloor\log_bx_{t-1}\rfloor-1}\mod{x_{t-1}}{p}$$
for all $t\leq g$, thus $s_b(x_g)=s_b(x_0)+g(b-1)=s_b(\tilde{x})+g(b-1)$ and $x_g\equiv x_0\mod{\tilde{x}}{p}$. Lastly, let $\alpha$ and $\beta$ be integers such that $b^{\alpha\omega}>x_g$ and $b^{\beta\omega}>\tau_b(p)$, and we define
$$x=x_g+\sum_{\iota=0}^{h-1}\tau_b(p)\cdot b^{(\iota\beta+\alpha)\omega}.$$
Now, $s_b(x)=s_b(x_g)+h\cdot s_b(\tau_b(p))=k$, $x\equiv x_g+\sum_{\iota=0}^{h-1}p\cdot b^{(\iota\beta+\alpha)\omega}\mod{-m^{-1}r}{p}$, and since every summand of $x$ is a distinct power of $b$ where the powers differ by at least $\omega$, we have $s_b(mx+r)=s_b(m)\cdot s_b(x)+s_b(r)=s_b(m)\cdot k+s_b(r)=dp$. Therefore, $mx+r$ is a $b$-Niven number due to the following observations.
\begin{itemize}
\item $mx+r\equiv m(-m^{-1}r)+r\mod{0}{p}$,
\item $d\mid(mx+r)$ since $d\mid m$ and $d\mid r$ by Lemma~\ref{lem:d|b-1}, and
\item $\gcd(p,d)=1$ since 
$p>b$ and $d\mid b-1$.
\end{itemize}
\end{proof}

\begin{lemma}\label{lem:carry}
Let $n$ be a nonnegative integer. For all nonnegative integers $y$, $s_b(yn)=ys_b(n)+z(b-1)$ for some integer $z$.
\end{lemma}

\begin{proof}
Note that for all nonnegative integers $n$, if $n=\sum_{i=0}^{\lfloor\log_bn\rfloor}\nu_b(n,i)b^i$, then $s_b(n)=\sum_{i=0}^{\lfloor\log_bn\rfloor}\nu_b(n,i)\mod{n}{b-1}$. Hence, $s_b(yn)\equiv yn\mod{ys_b(n)}{b-1}$. 
\end{proof}

\begin{proposition}\label{prop:multiple}
Let $d=\gcd(s_b(m),s_b(r),b-1)$. Then there exists a positive multiple $\overline{m}$ of $m$ such that $\gcd(s_b(\overline{m}),s_b(r))=d$.
\end{proposition}

\begin{proof}
Let $i_0$ be the smallest nonnegative integer such that $\nu_b(m,i_0)\neq0$. Then there exists a nonnegative integer $a\leq b-1$ such that $\nu_b(am,i_0)=\nu_b(a\cdot\nu_b(m,i_0),0)\geq\frac{b}{2}$. Next, if $\nu_b(am,i_0+1)\neq b-1$, then let $m'=am$; otherwise, let $m'=(b+1)am$ so that $\nu_b(m',i_0)=\nu_b(am,i_0)\geq\frac{b}{2}$ and
$$\nu_b(m',i_0+1)\equiv\nu_b(am,i_0+1)+\nu_b(am,i_0)\equiv b-1+\nu_b(am,i_0)\not\mod{b-1}{b}.$$
Furthermore, define $m''$ to be a multiple of $m'$ such that the leading digit of $m''$ in base-$b$ representation is at least $\frac{b}{2}$, i.e., $\nu_b(m'',\lfloor\log_bm''\rfloor)\geq\frac{b}{2}$. Let $m^*=b^2m''+m'$. Then $m^*$ is a multiple of $m$ such that $\nu_b(m^*,i_0)\geq\frac{b}{2}$, $\nu_b(m^*,i_0+1)\neq b-1$, and $\nu_b(m^*,\lfloor\log_bm^*\rfloor)\geq\frac{b}{2}$.

Let $x,y,z$ be integers such that $xs_b(r)+ys_b(m)+z(b-1)=d$. Define $y^*$ such that $m^*=y^*m$, and let $z^*$ be an integer such that $s_b(m^*)=y^*s_b(m)+z^*(b-1)$ by Lemma~\ref{lem:carry}. Letting $m^{**}=(b^{\lfloor\log_bm^*\rfloor-i_0}+1)m^*$, we see that $\nu_b(m^{**},\lfloor\log_bm^*\rfloor)=\nu_b(m^*,\lfloor\log_bm^*\rfloor)+\nu_b(m^*,i_0)-b$ and $\nu_b(m^{**},\lfloor\log_bm^*\rfloor+1)=\nu_b(m^*,i_0+1)+1\leq b-1$. Hence, $s_b(m^{**})=2s_b(m^*)-(b-1)=2y^*s_b(m)+(2z^*-1)(b-1)$. By Lemma~\ref{lem:frobenius}, there exist nonnegative integers $g$ and $h$ such that $gz^*+h(2z^*-1)\mod{z}{s_b(r)}$. Let $j$ be a nonnegative integer such that $gy^*+h(2y^*)+j\mod{y}{s_b(r)}$. Consider
\begin{align*}
\overline{m}=&\sum_{\iota=0}^{g-1}m^*b^{\iota(\lfloor\log_bm^*\rfloor+1)}+\sum_{\iota=0}^{h-1}m^{**}b^{\iota(\lfloor\log_bm^{**}\rfloor+1)+g(\lfloor\log_bm^*\rfloor+1)}\\
&+\sum_{\iota=0}^{j-1}mb^{\iota(\lfloor\log_bm\rfloor+1)+g(\lfloor\log_bm^*\rfloor+1)+h(\lfloor\log_bm^{**}\rfloor+1)}.
\end{align*}
By construction, $\overline{m}$ is a multiple of $m$ and
\begin{align*}
s_b(\overline{m})&=gs_b(m^*)+hs_b(m^{**})+js_b(m)\\
&=g\big(y^*s_b(m)+z^*(b-1)\big)+h\big(2y^*s_b(m)+(2z^*-1)(b-1)\big)+js_b(m)\\
&=(gy^*+h(2y^*)+j)s_b(m)+(gz^*+h(2z^*-1))(b-1)\\
&\equiv ys_b(m)+z(b-1)\mod{d}{s_b(r)}.
\end{align*}
Note that $d\mid s_b(\overline{m})$ since $d\mid s_b(m)$ and $d\mid(b-1)$. Therefore, $\gcd(s_b(\overline{m}),s_b(r))=d$.

\end{proof}
Combining Propositions~\ref{prop:existence} and \ref{prop:multiple}, we obtain the following theorem.

\begin{theorem}
Let $m$ and $r$ be positive integers. The arithmetic progression $\mathcal{S}_{m,r}$ contains infinitely many $b$-Niven numbers.
\end{theorem}
\begin{proof}
By Proposition~\ref{prop:multiple}, there exists a multiple $\overline{m}$ of $m$ such that $\gcd(s_b(\overline{m}),s_b(r),b-1)=\gcd(s_b(\overline{m}),s_b(r))$. Hence, by Proposition~\ref{prop:existence}, $\mathcal{S}_{\overline{m},r}$, and thus $\mathcal{S}_{m,r}$, contains at least one $b$-Niven number since $\mathcal{S}_{\overline{m},r}$ is a subset of $\mathcal{S}_{m,r}$. Let this $b$-Niven number be $\eta m+r$ for some nonnegative integer $\eta$. Applying the same argument on the arithmetic progression $\mathcal{S}_{m,(\eta+1)m+r}$ yields another $b$-Niven number, and our proof is complete by induction.
\end{proof}

\end{document}